\documentclass[a4paper, 11pt]{amsart}   
\usepackage{mathptmx, amssymb,amscd,latexsym, eulervm}   
\usepackage{amsmath}
\usepackage{amsthm} 
\usepackage{mathdots}
\usepackage[pagebackref,colorlinks=true,citecolor=violet,linkcolor=blue,urlcolor=blue]{hyperref}
\usepackage{color}
\usepackage[dvipsnames]{xcolor}
\usepackage[onehalfspacing]{setspace}
\usepackage{tabularx}
\usepackage{amsfonts,amsthm}
\usepackage{paralist}
\usepackage{aliascnt}
\usepackage{amscd}
\usepackage{blkarray}
\usepackage{mathbbol,bbm}
\usepackage{comment}
\usepackage{setspace}
\usepackage{multicol}
\usepackage[inner=2.5cm,outer=2.5cm, bottom=3.2cm]{geometry}
\usepackage{tikz, tikz-cd}
\usepackage{calligra,mathrsfs}
\usepackage{stmaryrd}
\usepackage{youngtab}
\usepackage{ytableau}
\usepackage{soul}
\usepackage{mathtools}
\usepackage{tikz}
\usetikzlibrary{matrix}
\usetikzlibrary{arrows,calc}
\allowdisplaybreaks
\newtheorem{theorem}{Theorem}[section]

\newaliascnt{headcor}{headthm}

\aliascntresetthe{headcor}

\newaliascnt{headconj}{headthm}

\aliascntresetthe{headconj}

\newaliascnt{corollary}{theorem}
\newtheorem{corollary}[corollary]{Corollary}
\aliascntresetthe{corollary}

\newaliascnt{claim}{theorem}

\aliascntresetthe{claim}

\newaliascnt{lemma}{theorem}
\newtheorem{lemma}[lemma]{Lemma}
\aliascntresetthe{lemma}

\newaliascnt{conjecture}{theorem}

\aliascntresetthe{conjecture}

\newaliascnt{proposition}{theorem}

\aliascntresetthe{proposition}

\AtBeginDocument{%
	\def\MR#1{}
}

\theoremstyle{definition}
\newaliascnt{definition}{theorem}
\newtheorem{definition}[definition]{Definition}
\aliascntresetthe{definition}

\newaliascnt{notation}{theorem}

\aliascntresetthe{notation}

\newaliascnt{example}{theorem}
\newtheorem{example}[example]{Example}
\aliascntresetthe{example}

\newaliascnt{examples}{theorem}

\aliascntresetthe{examples}

\newaliascnt{remark}{theorem}
\newtheorem{remark}[remark]{Remark}
\aliascntresetthe{remark}

\newaliascnt{question}{theorem}

\aliascntresetthe{question}

\newaliascnt{questions}{theorem}

\aliascntresetthe{questions}

\newaliascnt{problem}{theorem}

\aliascntresetthe{problem}

\newaliascnt{construction}{theorem}

\aliascntresetthe{construction}

\newaliascnt{setup}{theorem}

\aliascntresetthe{setup}

\newaliascnt{algorithm}{theorem}

\aliascntresetthe{algorithm}

\newaliascnt{observation}{theorem}

\aliascntresetthe{observation}

\newaliascnt{defprop}{theorem}

\aliascntresetthe{defprop}

\newaliascnt{fact}{theorem}

\aliascntresetthe{fact}

\newcommand{\bfu}{\mathbf{u}}
\newcommand{\bfv}{\mathbf{v}}
\newcommand{\bfm}{\mathbf{m}}
\newcommand{\bfn}{\mathbf{n}}
\newcommand{\bft}{\mathbf{t}}
\newcommand{\bfa}{\mathbf{a}}
\newcommand{\bfb}{\mathbf{b}}
\newcommand{\bfe}{\mathbf{e}}
\newcommand{\sP}{\mathscr{P}}
\newcommand{\sG}{\mathscr{G}}
\newcommand{\N}{\mathbb{N}}
\newcommand{\R}{\mathbb{R}}
\newcommand{\PP}{\mathbb{P}}

\newcommand{\Z}{\mathbb{Z}}
\newcommand{\Q}{\mathbb{Q}}
\newcommand{\ds}{\displaystyle}
\DeclareMathOperator{\rk}{rk}
\DeclareMathOperator{\St}{St}
\DeclareMathOperator{\conv}{conv}
\DeclareMathOperator{\supp}{supp}
\DeclareMathOperator{\cave}{cave}
\DeclareMathOperator{\Stal}{Stal}
\DeclareMathOperator{\Mob}{\textnormal{M\"ob}}
\def\equationautorefname~#1\null{(#1)\null}
\def\sectionautorefname~#1\null{Section #1\null}
\def\subsectionautorefname~#1\null{\S #1\null}

\begin{document}
\title{Three Combinatorial Algorithms for the Cave Polynomial of a Polymatroid}
\author{Anna Shapiro}
    \address{Department of Mathematics, North Carolina State University, Raleigh, NC 27695, USA}
	\email{arshapi4@ncsu.edu}

\subjclass[2020]{Primary: 52B40, 14C17. Secondary: 13H15}

\keywords{Polymatroids, M\"{o}bius function, cave polynomial, Snapper polynomial}

	\begin{abstract}
The cave polynomial of a polymatroid was recently introduced and used to study the syzygies of polymatroidal ideals. We study the combinatorial relationships between three formulas for the cave polynomial. As an application, we interpret the Snapper polynomial in terms of these three formulas.
	\end{abstract}

    \thanks{Anna Shapiro received support from NSF Grants DMS-2452179 and  DMS-2502321, as well as Simons Foundation Travel Support for Mathematicians Awards MPS-TSM-00007970 and MPS-TSM-00013551.}

\maketitle

\section{Introduction}
Polymatroids are a generalization of matroids, and objects of interest in combinatorial optimization.  For more about polymatroids, we point to \cite{Schrijver}. The cave polynomial of a polymatroid was recently introduced in \cite{CRMS} and used to settle two conjectures regarding polymatroids. The first was a conjecture of Bandari, Bayati, and Herzog from 2018, stating that homological shift ideals of polymatroidal ideals are again polymatroidal ideals. The second was a conjecture of Castillo, Cid-Ruiz, Mohammadi, and Monta\~no, stating that the M\"obius support of a polymatroid is a generalized polymatroid. The cave polynomial, which is relatively simple to compute given a polymatroid, connects and resolves both of these conjectures. 
\medskip

We study three known methods for computing the cave polynomial. While it is known that these three methods give the same result, prior proofs of the equivalences rely on deep machinery from algebra and algebraic geometry (see \autoref{sec:realizable} for a recollection of the realizable case and \autoref{sec:arbitrary} for the arbitrary case). The main results of this paper (\autoref{thm:equality}, \autoref{thm:mob value}, and \autoref{thm:stal mob}) give combinatorial relationships between the three methods. 
\medskip

We use $\bfe_i$ to denote the $i$th standard basis vector in $\N^p$. For any point $\bfn = (n_1,\ldots,n_p) \in \N^p$, we set $|\bfn| \coloneqq n_1+\cdots + n_p$, and $\bft^\bfn \coloneqq t_1^{n_1}\cdots t_p^{n_p}$. The base and independence polytopes of a polymatroid $\sP$ are denoted by $B(\sP)$ and $I(\sP)$, respectively. We use $\rk(\sP)$ to mean the rank of the polymatroid. A definition of polymatroids is given in \autoref{def:Mconvex}, and related concepts are defined in \autoref{sec:caves}.  

\begin{definition} 
Let $\sP$ be a polymatroid on $[p]\coloneqq \{1,2, \ldots, p\}$. 
\begin{enumerate}
\item The \textit{indicator function} of $\sP$ is the function $\mathbb{1}_{\sP}\colon \R^p \to \N$ given by $$\mathbb{1}_{\sP}(\bfn) = \begin{cases}
    1 & \text{ if }\bfn\in B(\sP),\\
    0 & \text{ if }\bfn \notin B(\sP).
\end{cases}$$
\item The \textit{cave polynomial} of $\sP$ is 
$$\cave_{\sP}(\bft) = \sum_{\substack{\bfn \in \N^p \\ |\bfn| = \rk(\sP)}}\mathbb{1}_{\sP}(\bfn)\prod_{i=1}^{p-1}\left(1-\max_{i<j}\{\mathbbm{1}_{\sP}(\bfn-\bfe_i+\bfe_j)\}t_i^{-1}\right)\bft^\bfn.$$
\end{enumerate}
\end{definition}

This definition of the cave polynomial is based on the combinatorial notion of caves, introduced in \cite[Section 5]{CCRMM}. There are three other polynomials of interest from the literature. We  compute an example of each of these polynomials in \autoref{ex:stal}, \autoref{ex:mob}, and \autoref{ex:box}. The first comes from the combinatorial ideas of caves and stalactites, introduced in \cite{CCRMM} (see \autoref{sec:caves} for more detail). 
\begin{definition}\label{def:Stal} Let $c_{\bfn}(\sP)$ denote the number of stalactites of $\sP$ containing $\bfn$. We define the \textit{stalactite polynomial} of $\sP$ to be  
\[\rm{Stal}_{\sP}(\bft) = \sum_{\bfn \in I(\sP)}(-1)^{\rk(\sP)-|\bfn|}c_\bfn(\sP)\bft^\bfn.\]
\end{definition}


The second comes from the Snapper polynomial of a matroid whose restriction polymatroid is $\sP$ (see \autoref{rk:Snapp}). 
The formula for the Snapper polynomial of a matroid was given in \cite[Corollary 7.5]{LLPP}. 
This in fact depends only on $\sP$, not on the matroid (see \cite[Corollary 2.9]{EL23}). 

\begin{definition}\label{def:box}
    The \textit{box polynomial} of $\sP$ is given by 
    \[\textrm{Box}_{\sP}(\bft) = \sum_{\bfn \in I(\sP)}\prod_{i=1}^p (t_i^{n_i}-\min\{1, n_i\}t_i^{n_i-1}).\]
\end{definition}
The coefficient $\min(1, n_i)$ ensures that, when $n_i=0$, we get a factor of 1 rather than 0. We note that when $n_i \ge 1$, the factor $t_i^{n_i} - \min(1, n_i)t_i^{n_i-1}$  is simply $t_i^{n_i} - t_i^{n_i-1} = t_i^{n_i-1}(t_i-1)$, which is the discrete analogue of the derivative $\frac{d}{dt_i}t_i^{n_i} = n_it_i^{n_i-1}$. The full product $\prod_i (t_i^{n_i}-\min\{1,n_i\}t_i^{n_i-1})$ is thus a mixed discrete derivative of $\mathbf{t}^{\mathbf{n}}$, one in each variable, and the box polynomial is the sum of these corner contributions over the independence polytope.

 
The third polynomial describes the class of a multiplicity-free variety $X$ in the Grothendieck ring $K(\PP)$ (where $\PP$ is a product of projective spaces, see \autoref{sec:realizable}). The coefficients describing this class as a linear combination of Schubert classes are particular values of the M\"obius function of a poset associated to the polymatroid $\sP$ (see \autoref{mobius section} for the precise definition). We use $\mu_P(-,-)$ to denote the M\"obius function of the poset, and write $\mu_\sP(\bfn) \coloneqq -\mu_P(\bfn, \hat{1})$. The following polynomial is similar to the characteristic polynomial of a poset (\cite[Section 3.10]{EC1}, although here each monomial corresponds to a single element, rather than all elements of a particular rank.

\begin{definition}
    The \textit{M\"obius polynomial} of $\sP$ is 
    \[\textrm{M\"ob}_{\sP}(\bft) = \sum_{\bfn\in I(\sP)} \mu_\sP(\bfn)\bft^\bfn.\]
\end{definition}


We can now state the main result of the paper. 
\begin{theorem}\label{thm:equality}
    The three polynomials above are all equal to the cave polynomial. That is, for any polymatroid $\sP$, we have 
    \[\cave_{\sP}(\bft) = \Stal_{\sP}(\bft) = \mathrm{Box}_{\sP}(\bft) = \Mob_{\sP}(\bft).\]
\end{theorem}

A certain class of polymatroids, called realizable polymatroids, are related to multiplicity-free varieties (in the sense of \cite{Brion}) in the following way: for any realizable polymatroid $\sP$, one can construct a multiplicity-free variety $X_\sP \subseteq \PP \coloneqq \PP^{m_1} \times \cdots \times \PP^{m_p} $ whose multidegree support is $\sP$ (see \cite[proof of Proposition 7.15]{CCRMM} for the construction). This construction is essential to prior understanding of the cave polynomial in the realizable case. We will give a brief summary of the proof of this result in the realizable case, and then in the arbitrary case. These proofs rely heavily on algebraic geometry. 

\subsection{The realizable case}\label{sec:realizable}

Let $\sP$ be a polymatroid and $X_\sP$ a multiplicity-free variety with $\sP$ as its multidegree support. Then the coefficients of the three polynomials defined above appear in two places:  as the coefficients of the multigraded Hilbert polynomial of $X_\sP$ written in terms of the binomial expressions $\binom{t_1+a_1}{a_1}\cdots\binom{t_p+a_p}{a_p}$, and  as the coefficients of the expression of $[X_\sP]$ as a linear combination of the classes of $\mathcal{O}_{\PP^{a_1} \times \cdots \times \PP^{a_p}}$  in the Grothendieck ring $K(\PP)$. In \cite[Theorem 6.9]{CCRMM}, it was shown that the coefficients of the stalactite polynomial of $\sP$ are also the Hilbert coefficients of the variety $X_\sP$. The proof uses a shelling order of the simplicial complex $\Delta(J)$ associated to the generic initial ideal $J = \mathrm{gin}(\mathfrak{P})$ of the $\N^p$-graded prime ideal $\mathfrak{P}$ associated to $X_\sP$. Conca and Tsakiris proved in \cite[Theorem 3.9]{ConcaTsakiris} that the multigraded Hilbert function is given by 
\[\sum_{\mathbf{n} \in I(\sP)}\prod_{i=1}^p\binom{t_i+n_i-1}{n_i}.\]

 Let $\mathfrak{b}\colon \Q[t_1,\ldots,t_p] \to \Q[t_1,\ldots,t_p]$ denote the polynomial map sending monomials to their corresponding products of binomial expressions, that is, $$\mathfrak{b}\colon t_1^{n_1}\cdots t_p^{n_p} \mapsto \binom{t_1+n_1}{n_1}\cdots\binom{t_p+n_p}{n_p}.$$

Then we have \[\sum_{\mathbf{n} \in I(\sP)}\prod_{i=1}^p\binom{t_i+n_i-1}{n_i}  = \mathfrak{b}(Box_{\sP}(\bft)),\]

so the coefficients of the multigraded Hilbert function are the same as those of the box polynomial.
 It was shown by Knutson in \cite[Theorem 3]{Knutson} that the coefficients describing the class of a multiplicity-free variety $X$ in the Grothendieck ring as a linear combination of Schubert classes are given by the M\"obius function of an order ideal in the Bruhat order. In \cite[Theorem 6.12(iii)]{CCRMM}, it was also proven that the Hilbert coefficients satisfy a M\"obius-like recurrence, i.e. that the stalactite polynomial is equal to the M\"obius polynomial. This proof uses properties of multiplicity-free varieties, namely their behavior under cuts by generic hyperplanes and projections.

\subsection{The arbitrary case}\label{sec:arbitrary}

Brion showed that any multiplicity-free variety $X_\sP$ has a flat degeneration to a variety $Y_\sP$ \cite{Brion}. As in \cite[Definition 2.7]{CRMS}, for any polymatroid $\sP$, we can define its dual polymatroidal ideal $J_\sP$, and then $Y_\sP = V(J_\sP)$. The generic initial ideal $J$ used in the proof that the stalactite coefficients are equal to the Hilbert coefficients coincides with $J_\sP$, so this is true in the arbitrary case as well. It is known that the cave polynomial of a polymatroid is related to the Snapper polynomial of a polymatroid (see \autoref{rk:Snapp}). Eur and Larson gave the following formula for the Snapper polynomial in \cite[Proposition 2.8]{EL23}:
\[\mathrm{Snapp}_\sP(t_1,\ldots,t_p) = \sum_{\bfn \in I(\sP)} \prod_{i=1}^p \binom{t_i+n_i-1}{n_i}.\]
This shows that the box polynomial is equal to the cave polynomial in the arbitrary case. The flat degeneration of $X_\sP$ to $Y_\sP$ also means that we can use $Y_\sP$ to compute the Hilbert polynomial, so the fact that the Hilbert coefficients are given by a M\"obius function is still true in the arbitrary case. 

\subsection{Combinatorial Approach} The primary goal of this paper is to further combinatorialize the proof of \autoref{thm:equality}. In particular, we give an explicit formula for the M\"obius function and show that it computes the coefficients of the formula given in ~\autoref{def:box}. 

\begin{theorem}\label{thm:mob value} Let $\bfm \le \bfn < \hat{1}$ in $P$. Then  \[\mu_P(\bfm,\bfn) = \begin{cases}
    (-1)^j & \text{if } \bfm = \bfn-\bfe_{i_1}-\cdots - \bfe_{i_j} \text{ for some distinct } i_1, \ldots, i_j \in [p],\\
    0 & \text{otherwise.}
\end{cases}\]
\end{theorem}

We also show that the number of stalactites containing a point is given by a M\"obius-like recurrence. 

\begin{theorem} \label{thm:stal mob}
    For each $\bfn \in \N^p$, the signed number of stalactites $c_\bfn'(\sP) \coloneqq (-1)^{\rk(\sP) - |\bfn|}c_\bfn(\sP)$ obeys the following recurrence: 
    $$c'_\bfn(\sP) = \begin{cases}
    1 & \text{if } \bfn \in B(\sP),\\
1-\ds\sum_{\bfm >\bfn} c'_\bfm(\sP) &\text{if } \bfn \in I(\sP) \setminus B(\sP),\\
0 & \text{if }\bfn \notin I(\sP). 
    \end{cases}
    $$

\noindent    That is, $c'_\bfn(\sP) = \mu_\sP(\bfn).$
\end{theorem}

One application of these results is to the Snapper polynomial. 

\begin{remark}\label{rk:Snapp}
The augmented $K$-ring $K(M)$ of a matroid was introduced by Larson, Li, Payne, and Proudfoot \cite{LLPP}. Let $M$ be a matroid on ground set $E$ with subsets $S_1, \ldots, S_p$ such that the restriction polymatroid is $\sP$ (with cage $\bfm$). That is, 
\begin{enumerate}
\item the subsets $S_1, \ldots, S_p$ form a partition of $E$,
\item $|S_i| = m_i$ for all $i \in [p]$,
\item The rank function of the matroid, $\rk_M:2^E \to \N$ is preserved by the product of symmetric groups $\mathfrak{S}_{S_1}\times \cdots \times \mathfrak{S}_{S_p}$, and 
\item For each $J \subseteq [p]$, we have $$\rk_\sP(J) = \rk_M\left(\bigcup_{j \in J}S_j\right).$$
\end{enumerate}

Then, following the construction in \cite{LLPP}, the Snapper polynomial of $\sP$ is 
\[\mathrm{Snapp}_\sP(t_1,\ldots,t_p) \coloneqq \chi\left(M, \mathcal{L}_{S_1}^{\otimes t_1} \otimes \cdots \otimes \mathcal{L}_{S_p}^{\otimes t_p}\right),\]
where $\chi(M,-)\colon K(M) \to \Z$ denotes the Euler characteristic map and $\mathcal{L}_{S_i}$ denotes the line bundle corresponding to the subset $S_i$. 

 It was also shown in \cite{CRMS} that $\mathrm{Snapp}_{\sP}(t_1,\ldots,t_p) = \mathfrak{b}(\cave_\sP(t_1,\ldots,t_p))$, where $\mathfrak{b}$ is, as above, the polynomial map sending monomials to their corresponding products of binomial expressions. Recall that Eur and Larson showed that $\mathrm{Snapp}_\sP(t_1,\ldots,t_p) = \mathfrak{b}(\mathrm{Box}_\sP(t_1,\ldots,t_p))$\cite[Proposition 2.8]{EL23}.
 \end{remark}
 
 From \autoref{thm:equality}, we get the following equalities. 

\begin{corollary} Let $\mathfrak{b}$ be the map defined above. We have

\[\mathrm{Snapp}_{\sP}(t_1,\ldots,t_p) 
    = \mathfrak{b}(\textnormal{M\"ob}_\sP(t_1,\ldots,t_p)) 
    = \mathfrak{b}(\Stal_\sP(t_1,\ldots,t_p)).\]
\end{corollary}

\section{Preliminaries}\label{sec:caves}
Here we state the relevant definitions for the remainder of the paper. We first establish some notation. For a positive integer $p$ let $[p]$ denote the set $\{1, 2, \ldots, p\}$. We use $\bfe_i$ to denote the $i$th standard basis vector in $\N^p$. Given a subset $J = \{j_1,\ldots,j_k\} \subseteq [p]$ we say $\bfe_J \coloneqq \bfe_{j_1}+\cdots+\bfe_{j_k}$. For any point $\bfn = (n_1,\ldots,n_p) \in \N^p$, we say $|\bfn| \coloneq n_1+\cdots + n_p$.

\begin{definition} A \textit{polymatroid} on $[p]$ with \textit{cage} $\bfm = (m_1,\ldots,m_p)$ is given by a function $\rk_{\sP}\colon 2^{[p]} \to \N$ satisfying the following conditions:
\begin{enumerate}
    \item $\rk_{\sP}(\varnothing) = 0$,
    \item $\rk_{\sP}(\{i\}) \le m_i$,
    \item If $I_1 \subseteq I_2 \subseteq [p]$, then $\rk_{\sP}(I_1) \le \rk_{\sP}(I_2)$, and
    \item For all $I_1, I_2 \subseteq [p]$, $\rk_{\sP}(I_1) + \rk_{\sP}(I_2) \ge \rk_\sP(I_1 \cup I_2) + \rk_{\sP}(I_1 \cap I_2)$. 
\end{enumerate}
Then $\rk_{\sP}$ is called the \textit{rank function of $\sP$,} and the \textit{rank of the polymatroid }is $\rk(\sP)\coloneqq \rk_{\sP}([p])$. 
\end{definition}

Under this definition, a \textit{matroid} is a polymatroid with cage $(1,\ldots,1)$. There is a second definition which will be more convenient for us to use. The two definitions were shown to be equivalent in \cite[Theorem 3.4]{HH2002}.  

\begin{definition}\label{def:Mconvex}
    We say a finite set $A \subseteq \N^p$ is \textit{homogeneous} if, for all $\bfa, \bfb \in A$, we have $|\bfa| = |\bfb|$. A \textit{polymatroid} is a finite homogeneous set of points $\sP \subseteq \N^p$ that is M-convex, i.e., 
    \begin{quote}
        For each $\bfu, \bfv \in \sP$ and $i \in [p]$ such that $u_i > v_i$, there is $j \in [p]$ such that $u_j < v_j$ and $\bfu-\bfe_i+\bfe_j \in \sP$. 
    \end{quote}
\end{definition}
If $\sP$ is actually a \textit{matroid}, then all entries of each $\bfn \in \sP$ are either $0$ or $1$. In fact, M-convexity boils down to the basis exchange property for matroids. 

\begin{example}\label{ex}
    We will use the polymatroid $\sP = \{(0,3), (1,2), (2,1)\}$ as a running example throughout the rest of the paper. 
\end{example}

\begin{remark}
The above definitions are related as follows: given a set of points $\sP$ forming a polymatroid on $[n]$, the corresponding rank function is given by $\rk_\sP(S) = \ds\max_{\bfn \in \sP}\left\{\sum_{i \in S}n_i\right\}$ for each $S \subseteq [n]$. Thus, the rank function corresponding to the polymatroid in \autoref{ex} is $$\rk_\sP(\varnothing) = 0, \quad \rk_\sP(\{1\}) = 2, \quad \rk_\sP(\{2\}) = 3,  \quad \rk_\sP(\{1,2\}) = 3.$$ 
\end{remark}

The cave polynomial is nonhomogeneous, so its support cannot be a polymatroid. Instead we use a generalization of a polymatroid which need not be homogeneous. 

\begin{definition}
    A \textit{generalized polymatroid} is a finite set $\mathscr{G} \subseteq \N^p$ whose homogenization is a polymatroid. Equivalently, following \autoref{def:Mconvex}, a generalized polymatroid is a set $\sG$ with the following two properties:
    \begin{enumerate}
        \item For each $\bfu, \bfv \in \sG$ with $u_i > v_i$, one of the following is true:
        \begin{enumerate}
            \item There exists an index $j \in [p]$ such that $u_j<v_j$ and $\bfu-\bfe_i+\bfe_j, \ \bfv+\bfe_i-\bfe_j \in \mathscr{G}$, or
            \item $|\bfu| > |\bfv|$ and $\bfu-\bfe_i, \ \bfv+\bfe_i \in \mathscr{G}$. 
        \end{enumerate}
        \item For each $\bfu, \bfv \in \mathscr{G}$ with $|\bfu|>|\bfv|$, there is $j \in [p]$ such that $u_j >v_j$ and both $\bfu-\bfe_j$ and $\bfv+\bfe_j$ are in $\mathscr{G}$. 
    \end{enumerate}
\end{definition}

\begin{definition}
    The \textit{independence polytope of $\sP$} is
$$I(\sP) \coloneqq \Big\{\bfn \in \R^p \bigm| \sum_{i\in I} n_i \le \rk(I) \text{ for all }I \subseteq [p]\Big\}.$$

The \textit{base polytope of $\sP$} is 
$$B(\sP) \coloneqq I(\sP) \cap \Big\{\bfn \in \R^p \bigm| \sum_{i=1}^pn_i = \rk(\sP)\Big\}.$$

    Equivalently, following the M-convexity definition, the base polytope is the convex hull of the points in the polymatroid, i.e. $B(\sP) = \conv(\sP)$. The independence polytope is the set of points below $B(\sP)$ in the nonnegative orthant, that is, $I(\sP) = (B(\sP) + \R^p_{\le0})\cap \R_{\ge0}^p$.
\end{definition}

Now we give the definitions of neighbors, stalactites, and caves. The construction of the cave polynomial mimics the construction of a cave as the union of stalactites. These definitions were first stated in \cite[Section 5]{CCRMM}. 

\begin{definition} 
Let $\sP$ be a polymatroid and $\bfu,\bfv \in B(\sP)$. 
\begin{enumerate}
    \item We say that $\bfu$ is a \textit{neighbor} of $\bfv$ in the $(-\ell,j)$ direction if $\bfu=\bfv-\bfe_{\ell}+\bfe_j$ for some $\ell,j \in [p]$. 
    \item Let $\ell_1,\ldots,\ell_m$ be distinct elements of $\supp(\bfu)$, that is, $u_{\ell_i} \ne 0$ for $i=1,\ldots,m$. Define
    $$\St(\bfu;\ell_1,\ldots,\ell_m) = \Big\{\bfu-\sum_{i\in J} \bfe_{\ell_i} \mid J \subseteq [m]\Big\}. $$
    \item Let $V \subseteq \sP$. Consider the set $J$ of indices $\ell$ such that there is $\mathbf{w}\in V$ which is a neighbor of $\bfu$ in the $(-\ell,j)$ direction for some $j \in [p]$. We define $\St(\bfu;V) = \St(\bfu;J)$. 
\end{enumerate}
\end{definition}

\begin{definition} Let $A$ be a finite set in $\N^p$.
Let $\max(A)$ denote $\ds\max_{\bfa \in A}\{|\bfa|\}$. The \textit{top} elements of $A$ are $A^{\rm{top}} \coloneqq \{\bfa \in A \mid |\bfa| = \max(A)\}$. 
For any $\bfb \in \N^p$, the \textit{truncation of $A$ at $\bfb$} is $A_\bfb \coloneqq \{\bfn \in A \mid \bfn\ge\bfb\}$.
\end{definition}

\begin{definition}
    A finite set $C \in \N^p$ is a \textit{cave} if each of the following conditions is satisfied:
    \begin{enumerate}
        \item The top elements $C^{\mathrm{top}}$ form a polymatroid. 
        \item Let $\prec$ denote any lexicographic order on $\N^p$ and order the elements of $C^{\mathrm{top}}$ as $\bfa_1 \prec \cdots \prec \bfa_{|C^{\mathrm{top}}|}$. Then 
        $$C = \bigcup_{i=1}^{|C^{\mathrm{top}}|} \St(\bfa_i;\{\bfa_1,\ldots,\bfa_{i-1}\}).$$
        \item For each nonzero $\bfb \in \N^p$, the truncation $C_\bfb$ is a generalized polymatroid. 
    \end{enumerate}
\end{definition}

 When we compute the stalactites $\St(\bfa_i;\{\bfa_1,\ldots,\bfa_{i-1}\})$ in order to find the coefficients of the stalactite polynomial, we can use any lex ordering $\prec$ to order the elements $\bfa_1 \prec \bfa_2 \prec \cdots \bfa_r$ of the polymatroid. The individual stalactites do depend on the choice of ordering, but the number of stalactites containing each point does not; this is illustrated in \autoref{ex:stal}. In \autoref{thm:stal=mob} we show that the stalactite polynomial is equal to the M\"obius polynomial, which does not depend on a lex ordering.   
 
 \begin{theorem}\label{thm:cave=stal}
     For any polymatroid $\sP$, we have $\cave_{\sP}(\bft) = \Stal_{\sP}(\bft)$. 
 \end{theorem}
 \begin{proof}
     Here we prove the statement for the stalactite polynomial computed using the standard lex order. In \autoref{thm:stal=mob} we will show that the stalactite polynomial is equal to the M\"obius polynomial, which does not depend on a lex ordering.  Order the elements of $B(\sP)$ as $\bfa_1 \prec \cdots \prec \bfa_r$ in the standard lex order. We show that the coefficient of $\bft^\bfm$ in the cave polynomial is equal to $(-1)^{\rk(\sP)-|\bfm|}c_\bfm(\sP)$. First we observe that the formula for the cave polynomial is cancellation-free, because the sign of each term is determined by $|\bfm|$, and the coefficient of $\bft^\bfm$ will have sign $(-1)^{\rk(\sP)-|\bfm|}$. The product $$\prod_{i=1}^{p-1}\left(1-\max_{i<j}\{\mathbbm{1}_{\sP}(\bfn-\bfe_i+\bfe_j)\}t_i^{-1}\right)\bft^\bfn$$ appears in the sum if and only if $\bfn \in B(\sP)$. The monomial $\bft^\bfm$ appears in the same product if and only if there exists a subset $J \subseteq [p]$ such that $\bfm = \bfn - \bfe_J$ and $\bfn$ has neighbors of the form $\bfn-\bfe_i+\bfe_j$, with $j > i$, for each $i \in J$. On the other hand, $\bfm \in \St(\bfa_k; \{\bfa_1, \ldots, \bfa_{k-1}\})$ if and only if there exists a subset $J \subseteq [p]$ such that $\bfm = \bfa_k-\bfe_J$ and $\bfa_k$ has neighbors of the form $\bfa_k-\bfe_i+\bfe_j \in \{\bfa_1, \ldots, \bfa_{k-1}\}$ for all $i \in J$. Notice that this means $\bfa_k-\bfe_i+\bfe_j \prec \bfa_k$ in the lex order, which is equivalent to the condition that $i < j$. So the coefficients of the two polynomials are equal. 
 \end{proof}

\begin{example}\label{ex:stal}
    For $\sP = \big\{(0,3), (1,2), (2,1)\big\}$, we can use the ordering $(0,3) \prec (1,2) \prec (2,1)$ and then we have 
    \begin{align*}
        \St\big((0,3);\varnothing\big) &= \{(0,3)\}.\\
        \St\big((1,2);\{(0,3)\}\big) &= \{(1,2), (0,2)\}.\\
        \St\big((2,1);\{(0,3), (1,2)\}\big) &= \{(2,1), (1,1)\}.
    \end{align*}
    So $\Stal_{\sP}(t_1,t_2) = t_2^3+t_1t_2^2-t_2^2+t_1^2t_2-t_1t_2$. 
    If we instead choose the ordering $(2,1) \prec (1,2) \prec (0,3)$, we get the following stalactites:
    \begin{align*}
        \St\big((2,1);\varnothing\big) &= \{(2,1)\}.\\
        \St\big((1,2);\{(2,1)\}\big) &= \{(1,2), (1,1)\}.\\
        \St\big((0,3);\{(2,1), (1,2)\}\big) &= \{(0,3), (0,2)\}.
    \end{align*}
Though the individual stalactites are different under this ordering, the stalactite polynomial is the same. By the definition of the cave polynomial, 
    \begin{align*}
        \cave_\sP(t_1,t_2) &= \sum_{\substack{\bfn \in \N^p \\ |\bfn| = \rk(\sP)}}\mathbb{1}_{\sP}(\bfn)\left(1-\mathbbm{1}_{\sP}(\bfn-\bfe_1+\bfe_2)t_1^{-1}\right)t_1^{n_1}t_2^{n_2}\\
        &= \mathbb{1}_{\sP}((0,3))\left(1-\mathbbm{1}_{\sP}((-1,4))t_1^{-1}\right)t_2^{3} + 
        \mathbb{1}_{\sP}((1,2))\left(1-\mathbbm{1}_{\sP}((0,3))t_1^{-1}\right)t_1^1t_2^2 \\
       & \quad  + \mathbb{1}_{\sP}((2,1))\left(1-\mathbbm{1}_{\sP}((1,2))t_1^{-1}\right)t_1^2t_2^1
        +\mathbb{1}_{\sP}((3,0))\left(1-\mathbbm{1}_{\sP}((2,1))t_1^{-1}\right)t_1^3\\
        &=t_2^3 + (1-t_1^{-1})t_1t_2^2 + (1-t_1^{-1})t_1^2t_2 +0\\
        &=t_2^3+t_1t_2^2-t_2^2+t_1^2t_2-t_1t_2.
    \end{align*}
\end{example}

\section{Formula for the M\"obius Function of a polymatroid}\label{mobius section}

In this section, we define and study the M\"obius function of a polymatroid. In particular, we prove an explicit formula for this M\"obius function which allows us to show that $\textrm{Box}_{\sP}(\bft) = \textrm{M\"ob}_{\sP}(\bft)$. We define a poset related to each polymatroid $\sP$. For background on M\"obius functions of posets we refer the reader to \cite[Section 3.7]{EC1}. Let $P$ be the poset with underlying set $I(\sP) \cap \N^p$ and a maximal element $\hat{1}$ adjoined, with relations given by the componentwise order. That is, $$\bfa = (a_1,\ldots, a_p) \le \bfb = (b_1, \ldots, b_p) \text{ if } a_i \le b_i \text{ for each } i \in [p].$$ Then the M\"obius function of $P$ is the function $\mu_P\colon \mathrm{Int}(P) \to \Z$ on the set of closed intervals of $P$, defined by the relation 
$$\mu_P(\bfm,\bfn) \coloneqq \begin{cases}
    1 &\text{ if }\bfm=\bfn,\\
    -\ds\sum_{\bfm \le \bfa < \bfn}\mu_P(\bfm, \bfa) &\text{ if }\bfm < \bfn,\\
    0 & \text{ if } \bfm \not\le \bfn.
\end{cases}$$

The only M\"obius values which appear as coefficients in the M\"obius polynomial are those of the form $\mu_P(\bfn, \hat{1})$. In fact the coefficient of $\bft^\bfn$ in $\Mob_{\sP}(\bft)$ is $-\mu_P(\bfn, \hat{1})$. For this reason we use the notation  $\mu_{\sP}(\bfn) \coloneqq -\mu_P(\bfn, \hat{1})$. From the relation above we get 
\begin{equation}\label{eqn}\mu_{\sP}(\bfn) = 
\begin{cases}
1 & \text{if } \bfn \in B(\sP),\\
1-\ds\sum_{\bfm >\bfn} \mu_{\sP}(\bfm) &\text{if } \bfn \in I(\sP) \setminus B(\sP),\\
0 & \text{if }\bfn \notin I(\sP). 
\end{cases}
\end{equation}

In order to compare the M\"obius values to the coefficients of the box polynomial, we give an explicit formula for this M\"obius function. 
\begin{theorem}\label{thm:mobius intervals} Let $\bfm \le \bfn < \hat{1}$ in $P$. Then  \[\mu_P(\bfm,\bfn) = \begin{cases}
    (-1)^j & \text{if } \bfm = \bfn-\bfe_{i_1}-\cdots - \bfe_{i_j} \text{ for some distinct } i_1, \ldots, i_j \in [p],\\
    0 & \text{otherwise.}
\end{cases}\]
\end{theorem}
\begin{proof}
Let $P$ be the poset described above, and let $\bfm \le \bfn < \hat{1}$ in $P$. Let $\ell$ denote the length of the interval $[\bfm, \bfn] \subseteq P$. We will induct on $\ell$. If $\ell=0$ then $\bfm=\bfn$, so $\mu_P(\bfm, \bfn) = 1 = (-1)^0$. If $\ell = 1$, then $\bfn = \bfm + \bfe_i$ for some $i \in [p]$. Substituting, we obtain $$\mu_P(\bfm, \bfn) = -\mu_P(\bfm, \bfm) = -1 = (-1)^1.$$ Now suppose $\ell \ge 1$ and that the formula is valid for intervals of length less than $\ell$. We have two cases. First suppose that $\bfm = \bfn - \bfe_{i_1} - \cdots - \bfe_{i_{\ell}}$ for some $\{i_1, \ldots, i_{\ell}\} \subseteq [p]$. Then  each element $\bfa$ in the interval $[\bfm, \bfn]$ is of the form $\bfa = \bfm + \bfe_J$ where $J \subseteq \{ i_1,\ldots,i_{\ell}\}$. Thus by induction,
\[\mu_P(\bfm, \bfn) = -\sum_{\bfm \le \bfa < \bfn}\mu_P(\bfm, \bfa) = -\sum_{J \subset \{i_1,\ldots,i_\ell\}}\mu_P(\bfm, \bfm+\bfe_J) = -\sum_{J \subset \{i_1,\ldots,i_\ell\}}(-1)^{|J|}=-\sum_{j=0}^{\ell-1}(-1)^j\binom{\ell}{j}.\]

Since the alternating sum of binomial coefficients is equal to zero, we have 

\[-\sum_{j=0}^{\ell-1}(-1)^j\binom{\ell}{j} =(-1)^\ell\binom{\ell}{\ell} = (-1)^{\ell}, \]
which concludes the first case. In the second case, we assume $\bfm \ne \bfn-\bfe_{i_1}-\cdots-\bfe_{i_\ell}$ for any subset $\{i_1, \ldots, i_\ell\} \subseteq [p]$. Let $A = \{i \in [p] \mid \bfm+\bfe_i < \bfn\}$. Notice that $|A| < \ell$. Let $B = \{\bfm + \bfe_J \in P \mid J \subseteq A\}$. If $\bfm \le \bfa < \bfn$ and $\bfa \notin B$, then $\mu_P(\bfm, \bfa) = 0$ by induction. So we have 
\[\mu_P(\bfm, \bfn) = -\sum_{\bfm\le\bfa<\bfn}\mu_P(\bfm, \bfa) = -\sum_{\bfa\in B}\mu_P(\bfm, \bfa) = -\sum_{J \subseteq A} (-1)^{|J|} = -\sum_{j=0}^{|A|} (-1)^j\binom{|A|}{j} = 0. \]
This completes the proof.
\end{proof}

\begin{example}\label{ex:mob}
    The polymatroid $\sP = \big\{(0,3), (1,2), (2,1)\big\}$ corresponds to the poset with the following Hasse diagram and M\"obius values: 
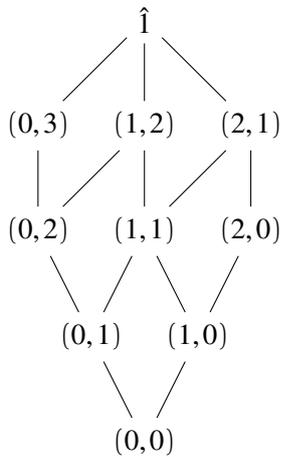
    \begin{multicols}{2}
$\begin{tikzpicture}[scale=0.7]
  \node (max) at (0,4) {$\hat{1}$};
  \node (a) at (-2,2) {$(0,3)$};
  \node (b) at (0,2) {$(1,2)$};
  \node (c) at (2,2) {$(2,1)$};
  \node (d) at (-2,0) {$(0,2)$};
  \node (e) at (0,0) {$(1,1)$};
  \node (f) at (2,0) {$(2,0)$};
  \node (g) at (-1,-2) {$(0,1)$};
  \node (h) at (1,-2) {$(1,0)$};
  \node (min) at (0,-4) {$(0,0)$};
  \draw (min) -- (g) -- (d) -- (a) -- (max);
  \draw (min) -- (h) -- (f) -- (c) -- (max);
  \draw (g) -- (e) -- (b) -- (max);
  \draw (h) -- (e) -- (c);
  \draw (d) -- (b);
\end{tikzpicture}$
\columnbreak 
\begin{align*} 
    \mu_\sP\big((0,3)\big) &= 1\\
    \mu_\sP\big((1,2)\big) &= 1\\
    \mu_\sP\big((2,1)\big) &= 1\\
    \mu_\sP\big((0,2)\big) &= -1\\
    \mu_\sP\big((1,1)\big) &= -1\\
    \mu_\sP\big((2,0)\big) &= 0\\
    \mu_\sP\big((0,1)\big) &= 0\\
    \mu_\sP\big((1,0)\big) &= 0\\
    \mu_\sP\big((0,0)\big) &= 0
\end{align*}
\end{multicols}

Therefore the M\"obius polynomial of $\sP$ is $\Mob_\sP(t_1,t_2) = t_2^3+t_1t_2^2+t_1^2t_2-t_2^2-t_1t_2$. 
\end{example}

As a consequence of \autoref{thm:mobius intervals}, we get the following result. 
\begin{theorem}\label{thm:box=mob}
    For any polymatroid $\sP$, we have $\mathrm{Box}_\sP(\bft) = \Mob_\sP(\bft)$. 
\end{theorem}
\begin{proof}
    We will compare coefficients of the two polynomials. First, notice that, for a fixed $\bfn \in I(\sP)$, the coefficient of $\bft^{\mathbf{m}}$ in the product $$\prod_{i=1}^p \big(t_i^{n_i}-\min\{1, n_i\}t_i^{ n_i-1}\big)$$ is equal to $(-1)^{|J|}$ if $\bfm = \bfn-\bfe_J \text{ for some } J \subseteq [p]$ and $0$ otherwise. That is, the coefficient of $\bft^{\mathbf{m}}$ in the product is $\mu_P(\bfm, \bfn)$. Thus when we sum over the independence polytope, we get that the coefficient of $\bft^\bfm$ in $\mathrm{Box}_\sP(\bft)$ is $\sum_{\bfn \in I(\sP)}\mu_P(\bfm, \bfn)$. The coefficient of $\bft^\bfm$ in the M\"obius polynomial is $\mu_{\sP}(\bfm)$, by definition. We have $$\mu_\sP(\bfm) = -\mu_P(\bfm, \hat{1}) = \sum_{\bfm \le \bfn < \hat{1}} \mu_P(\bfm, \bfn).$$ Since $\mu_P(\bfm, \bfn) = 0$ for $\bfn \in I(\sP)$ with $\bfm \not\le \bfn$, this is the same as $\sum_{\bfn \in I(\sP)}\mu_P(\bfm, \bfn)$. Hence the two polynomials have the same coefficients, so they are equal.

\end{proof}

\begin{example}\label{ex:box}
    For $\sP = \big\{(0,3), (1,2), (2,1)\big\}$, we have 
    \begin{align*}\mathrm{Box}_\sP(t_1,t_2) &= \sum_{\bfn \in I(\sP)}\big(t_1^{n_1}-\min\{1, n_1\}t_1^{n_1-1}\big)\big(t_2^{n_2}-\min\{1, n_2\}t_2^{ n_2-1}\big)\\
    &= (1)(t_2^3-t_2^2)+(t_1-1)(t_2^2-t_2) + (t_1^2-t_1)(t_2-1) + (1)(t_2^2-t_2) + (t_1-1)(t_2-1)\\ & \quad + (t_1^2-t_1)(1) + (1)(t_2-1) + (t_1-1)(1) + (1)(1)\\
    &= t_2^3-t_2^2 + t_1t_2^2-t_1t_2+t_1^2t_2.
    \end{align*}
\end{example}

\section{Stalactites are counted by a M\"obius function}
In order to prove that $\rm{Stal}_{\sP}(\bft) =  \textrm{M\"ob}_{\sP}(\bft)$, we examine the coefficients of the stalactite polynomial. That is, we show that the number of stalactites containing some $\bfn \in I(\sP)$ is counted by the M\"obius function. Let $c'_\bfn(\sP)$ denote the signed number of stalactites of the polymatroid $\sP$ containing $\bfn$, i.e. $$c'_\bfn(\sP) = (-1)^{\rk(\sP)-|\bfn|}\#\Big\{i \mid \bfn \in \St\big(\bfa_i;\{\bfa_1,\ldots,\bfa_{i-1}\}\big)\Big\}.$$ Then $c'_\bfn(\sP)$ is also the coefficient of $\bft^\bfn$ in the stalactite polynomial. We show that, for each $\bfn \in I(\mathscr{P})$, $c'_\bfn(\mathscr{P}) = \mu_{\mathscr{P}}(\bfn)$. We first prove two lemmas that will simplify the proof of \autoref{thm:stal=mob}. 

\begin{lemma}\label{lem:lem1}
    Let $\mathscr{P}$ be a polymatroid (as in \autoref{def:Mconvex}) and $\bfn \in I(\mathscr{P})$. Then $\mathscr{P}_{\bfn} \coloneqq \mathscr{P} \cap (\bfn + \N^p)$ is a polymatroid (called the truncation of $\sP$ at $\bfn$).  
\end{lemma} 

\begin{proof}
    Since $\sP_{\bfn}$ is a subset of the finite homogeneous set $\sP$, we know that $\sP_{\bfn}$ is finite and homogeneous as well. Thus, it suffices to check M-convexity of $\sP_{\bfn}$. Suppose $\bfu,\bfv \in \sP_{\bfn}$ with $u_i > v_i$. Then the same is true in $\sP$, so by M-convexity of  $\mathscr{P}$ there is an index $j \in [p]$ such that $u_j<v_j$ and $\bfu-\bfe_i+\bfe_j \in \mathscr{P}$. All we need to check is that $ \bfu-\bfe_i+\bfe_j \in \sP_{\bfn}$, i.e., that $\bfu-\bfe_i+\bfe_j \ge \bfn$ in the componentwise order. Since $\bfu \in \sP_{\bfn}$ by assumption, we know $\bfu \ge \bfn$. The only entry we need to check is the $i$th. Since $\bfv \in \sP_{\bfn}$ and $u_i > v_i$, we have that $u_i-1 \ge v_i \ge n_i$. Thus $\bfu-\bfe_i+\bfe_j \ge \bfn$, as we wanted. 
\end{proof}

\begin{lemma}\label{lem:lem2}
    Let $\mathscr{P}$ and $\sP_{\bfn}$ be as above. Then for all $\bfm \in \sP_{\bfn}$, we have $c'_\bfm(\mathscr{P}) = c'_\bfm(\sP_{\bfn})$. 
\end{lemma} 
\begin{proof}
    Let $\bfm \in \sP_{\bfn}$. We show that the stalactites which contain $\bfm$ in $\sP$ also contain $\bfm$ in $\sP_{\bfn}$. First, note that $\sP_{\bfn} \coloneqq \sP \cap (\bfn + \N^p)$ must contain $ \sP \cap (\bfm + \N^p)$. Therefore, in order to show that an element of $\sP$ is in $\sP_{\bfn}$, it suffices to show that the element is greater than or equal to $\bfm$ in the componentwise order. Suppose that $\bfm \in St(\bfa_i;\{\bfa_1,\ldots,\bfa_{i-1}\})$ in $\mathscr{P}$. Then $\bfm = \bfa_i-\bfe_J$ for some $J \subseteq [p]$, so $\bfa_i \ge \bfm$ and thus $\bfa_i \in \sP_\bfn$. The only neighbors of $\bfa_i$ which must be in $\sP_{\bfn}$ in order for $\bfm$ to be in the stalactite in $\sP_{\bfn}$ are those in the $(-\ell,j)$ directions for $\ell \in J$. Each such neighbor is of the form $\bfa_i - \bfe_\ell+\bfe_j$. We write $\bfa_i = (a_{i,1},\ldots, a_{i, p})$. Since $\bfa_i \ge \bfm$, we need only show that $a_{i,\ell}-1 \ge m_\ell$. In fact, since $\ell \in J$, we have $m_\ell = a_{i,\ell}-1$. Therefore $\bfa_i$ and its neighbors in the necessary directions are all in $\sP_{\bfn}$, so $\bfm$ is in the same number of stalactites in $\sP_{\bfn}$ as it is in $\sP$. 
\end{proof}

We can now prove the main result of this section. 
\begin{theorem}\label{thm:stal=mob}
    Let $\sP$ be a polymatroid. Then $\Stal_\sP(\bft) = \Mob_\sP(\bft)$. 
\end{theorem}
\begin{proof}
We show that the coefficients $c'_\bfn(\sP)$ of the stalactite polynomial obey the same relation as $\mu_\sP(\bfn)$ (see \autoref{eqn}). For any $\bfa \in B(\sP)$, let $\St(\bfa) \coloneqq \St(\bfa; \{\bfb \in B(\sP) \mid \bfb \prec \bfa\})$, where $\prec$ denotes a lex ordering. If $\bfn \in B(\sP)$, then $\bfn \in \St(\bfn)$ and $\bfn$ cannot be in any other stalactites, so $c'_\bfn(\sP) = 1$. If $\bfn \notin I(\sP)$, then $\bfn$ is in no stalactites, since each element of a stalactite is of the form $\bfu-\bfe_J$, for some $\bfu \in B(\sP)$ and $J \subseteq [p]$. Thus in this case $c'_\bfn(\sP) = 0$. It remains to show that if $\bfn \in I(\sP) \setminus B(\sP)$, then $c'_\bfn(\sP) = 1-\sum_{\bfm>\bfn}c_\bfm'(\sP)$, i.e. $\sum_{\bfm\ge\bfn}c_\bfm'(\sP) = 1$. Let $\sP_{\bfn}$ be the truncation of $\sP$ at $\bfn$. Then, since $I(\sP_\bfn) = I(\sP) \cap (\bfn + \R_{\ge 0}^p)$, and $c_{\bfm}'(\sP) = 0$ for $\bfm \notin I(\sP)$, the sum can be written over the independence polytope of $\sP_\bfn$. By \autoref{lem:lem2}, $\sum_{\bfm \in I(\sP_\bfn)}c_\bfm'(\sP) = \sum_{\bfm \in I(\sP_{\bfn})}c_\bfm'(\sP_{\bfn})$. Thus it suffices to show that for any polymatroid $\sP$, we have $\sum_{\bfn \in I(\sP)} c'_\bfn(\sP) = 1$. 

Let $c'_{\bfm, \bfu}(\sP)$ denote the contribution of $\St(\bfu)$ to $c'_\bfm$, that is, $$c'_{\bfm, \bfu}(\sP) \coloneqq \begin{cases}
    1 &\text{if } \bfm \in \St(\bfu) \text{ and } |\bfu| - |\bfm| \text{ is even,}\\
    -1 &\text{if } \bfm \in \St(\bfu) \text{ and } |\bfu| - |\bfm| \text{ is odd,}\\
    0 &\text{if } \bfm \notin \St(\bfu). 
\end{cases}$$

Then $c'_\bfm(\sP) = \sum_{\bfu \in B(\sP)} c'_{\bfm, \bfu}(\sP)$. Therefore 
\[\sum_{\bfn \in I(\sP)} c'_\bfn(\sP) = \sum_{\bfn \in I(\sP)}\sum_{\bfu \in B(\sP)} c'_{\bfn, \bfu}(\sP) = \sum_{\bfu \in B(\sP)}\sum_{\bfn \in I(\sP)} c'_{\bfn, \bfu}(\sP).\]
Let $B(\sP) = \bfa_1 \prec \bfa_2 \prec \cdots \prec \bfa_n$. We show that $\sum_{\bfn \in I(\sP)} c'_{\bfn, \bfa_1}(\sP) = 1$ and $\sum_{\bfn \in I(\sP)} c'_{\bfn, \bfa_i}(\sP) = 0$ for $i > 1$.
First, $\St(\bfa_1) = \St(\bfa_1; \varnothing) = \bfa_1$, so $c'_{\bfa_1,\bfa_1}(\sP) = 1$, and $c'_{\bfm, \bfa_1}(\sP) = 0$ for all other $\bfm \in I(\sP)$. Thus $\sum_{\bfn \in I(\sP)} c'_{\bfn, \bfa_1}(\sP) = 1$. 

Now let $i>1$ and consider $\St(\bfa_i)$. First we show that $\bfa_i$ has at least one neighbor in $\{\bfa_1, \ldots, \bfa_{i-1}\}$, i.e. there is $j<i$ such that $\bfa_j = \bfa_i - \bfe_{\ell}+\bfe_k$. Since $\bfa_1 \prec \bfa_i$, there is some $r \in [p]$ with $a_{1,r} < a_{i,r}$. Pick the least such $r$. Then $a_{1,q} = a_{i,q}$ for all $q < r$. Now by M-convexity of $\sP$, there exists $s \in [p]$ such that $a_{1,s} > a_{i,s}$ and $\bfa_i - \bfe_r + \bfe_s \in \sP$. Notice that $s > r$, and thus $\bfa_i-\bfe_r+\bfe_s \prec \bfa_i$. So $\bfa_i$ has at least one neighbor in $\{\bfa_1, \ldots, \bfa_{i-1}\}$. Let $J$ be the set of directions in which $\bfa_i$ has neighbors in $\{\bfa_1, \ldots, \bfa_{i-1}\}$. Then 
\[\sum_{\bfn \in I(\sP)} c'_{\bfn, \bfa_i}(\sP) = \sum_{\bfn \in \St(\bfa_i)} c'_{\bfn, \bfa_i}(\sP) = \sum_{J' \subseteq J} (-1)^{|J'|}.\]
Since $J$ is nonempty, this sum is equal to 0, as desired. 

\end{proof}

By combining \autoref{thm:cave=stal}, \autoref{thm:box=mob}, and \autoref{thm:stal=mob}, we obtain the main result: 

\begin{theorem}
    For any polymatroid $\sP$, we have 
    \[\cave_{\sP}(\bft) = {\rm Stal}_{\sP}(\bft) = {\rm Box}_{\sP}(\bft) = \text{\rm M\"ob}_{\sP}(\bft).\]
\end{theorem}

\section*{Acknowledgements}
The author would like to thank Yairon Cid-Ruiz and Jacob Matherne for helpful feedback and many productive discussions, and the anonymous reviewers for their thoughtful comments and suggestions.

\bibliographystyle{amsalpha}
\bibliography{references.bib}

\end{document}